\documentclass[12pt]{amsart}

\usepackage[all,arc]{xy}
\usepackage{enumerate, amsmath, amsthm, amsfonts, amssymb, xy,  mathrsfs, graphicx, paralist, fancyvrb,ytableau}
\usepackage{xcolor}
\usepackage[margin=1in]{geometry} 
\usepackage[bookmarks, colorlinks=true, linkcolor=blue, citecolor=blue, urlcolor=blue]{hyperref}

\newtheorem{thm}{Theorem}[section]

\newtheorem{prop}[thm]{Proposition}
\newtheorem{lem}[thm]{Lemma}

\theoremstyle{definition}

\newtheorem{exmp}[thm]{Example}

\theoremstyle{remark}
\newtheorem{rmk}[thm]{Remark}

\makeatletter
\let\c@equation\c@thm
\makeatother

\newcommand{\cE}{\mathcal{E}}

\newcommand{\cO}{\mathcal{O}}

\newcommand{\cQ}{\mathcal{Q}}

\newcommand{\cR}{\mathcal{R}}

\newcommand{\cS}{\mathcal{S}}

\newcommand{\cT}{\mathcal{T}}

\newcommand{\bZ}{\mathbf{Z}}

\newcommand{\dsum}{\bigoplus}

\makeatletter
\newcommand{\RN}[1]{%
  \textup{\uppercase\expandafter{\romannumeral#1}}%
  }
\makeatother

\numberwithin{equation}{section}
\newcommand{\arxiv}[1]{\href{http://arxiv.org/abs/#1}{{\tt arXiv:#1}}}

\title{Minimal Set of Generators of Ideals Defining Nilpotent Orbit Closures}
\author{Hang Huang}
\date{September 2018}

\begin{document}

\begin{abstract}
Over a field of characteristic $0$, we construct a minimal set of generators of the defining ideals of closures of nilpotent conjugacy class in the set of $n \times n$ matrices. This modifies a conjecture of Weyman and provides a complete answer to it.
\end{abstract}

\maketitle

\section{Introduction}
Let $X$ be the set of $n \times n$ matrices over a field $K$ of characteristic $0$. We consider $\cO_{\mu}$, which is the conjugacy class of nilpotent matrices in $X$ with Jordan blocks of size $\mu_1,\ldots,\mu_s$. Here $\mu$ is a partition of $n$. Let $X_{\mu}$ be the Zariski closure of $\cO_{\mu}$. It is called the nilpotent orbit closure variety. When $\mu = (n)$, Kostant showed in his paper \cite{K} that the radical ideal of $X_{\mu}$ is generated by the invariants of the conjugation action of $GL(n)$ on $X$. In \cite{DP}, DeConcini and Procesi conjectured a generating set of ideals defining $X_{\mu}$. Later in \cite{T}, Tanisaki proposed a simpler generating set. It was generalized to rank varieties by Eisenbud and Saltman in \cite{ES}. In 1989, Weyman showed all these conjectures hold using geometric methods in \cite{Weyman}. In the same paper, he conjectured a minimal generating set for these ideals. Later in \cite{weyman2}, he showed in details that his conjecture is valid for all rectangular partitions. However, in 2008, in the process of searching for a minimal generating set for De Concini-Procesi ideals of schematic intersections $X_{\mu} \cap D$ where $D$ is the set of diagonal matrices, Biagioli, Faridi and Rosas found counterexamples of Weyman's conjecture in \cite{BFR}. They gave a modified version of Weyman's conjecture which consists of questions about redundant generators.

In this paper we completely solve Weyman's original question about minimal generating set of ideals defining $X_{\mu}$. And therefore we also provide answers to some of the questions in \cite{BFR}.

Our paper is organized as follows. In Section~{\ref{sec:prelim}}, we provide necessary background and summary of Weyman's geometric methods used in \cite{Weyman}. This overlaps a lot with \cite[Section 0]{weyman2}. However, we put it here for the sake of completeness. In Section~{\ref{sec:main}}, we prove our main theorem.

\subsection*{Acknowledgements.} 

The author would like to send thanks to Steven Sam for bringing a related question which leads to this paper. The author would also like to send special thanks to him for all the fruitful conversations during the writing and conception of this work.

\section{Preliminaries and Notations} \label{sec:prelim}

\subsection{Nilpotent Orbit Closures}

Let $K$ be the underlying field. Throughout the paper, we assume the characteristic of $K$ is $0$. Let $E$ be a vector space of dimension $n$ over $K$. We choose a basis $\{ e_1,e_2,\ldots,e_n \}$ of $E$ and write the elements of $X$ in the form
\[
\phi = \sum_{i,j = 1}^{n} \phi_{i,j} e_j^* \otimes e_i
\]
where $\phi \colon E \rightarrow E$. Using this, we identify $X$ with $Hom_K (E,E) = E^* \otimes E$.

The general linear group $GL(n)$ acts by conjugation on $X$. Fixing a partition $\mu = (\mu_1,\ldots,\mu_s)$, we denote the conjugacy class of nilpotent matrices in $X$ with Jordan blocks of sizes $\mu_1,\ldots,\mu_s$ by $\cO_{\mu}$ and its closure in $X$ by $X_{\mu}$. In other words, if we denote the dual partition of $\mu$ by $\mu^T$, then $\cO_{\mu}$ consists of all endomorphisms $\phi$ such that dim Ker $\phi^i = \mu_1^T + \ldots + \mu_i^T$. Here $\mu_i^T$ is defined by $\mu_i^T = \# \{ j \mid \mu_j \geq i \}$. In addition, we denote $l(\mu) = $ the largest $i$ such that $\mu_i \neq 0$. We also denote the radical ideal of $X_{\mu}$ by $J_{\mu}$. For an arbitrary endomorphism $\Phi$, the coordinate functions $\Phi_{i,j}$ are the generators of the coordinate ring of $X$, which we denote by $A = K[X] = Sym(E \otimes E^*)$.

Let $I,J$ be two subsets of $\{1,\ldots,n\}$ of cardinality $r$. Let $(I|J)$ denote the determinant of the $r \times r$ submatrix of $\Phi$ with the rows from $I$ and columns from $J$. We also denote $S_{\mu}$ the Schur module corresponding to the highest weight $\mu$ for the group $GL(n)$. Here $\mu = (\mu_1, \ldots, \mu_n)$ is a dominant integral weight where $\mu_i \in \bZ$ and $\mu_1 \geq \ldots \geq \mu_n$.

Fix a number $p$ such that $1 \leq p \leq n$. We can look at the linear span of $p \times p$ minors of the matrix $\Phi$, which can be identified with $\bigwedge^p E \otimes \bigwedge^p E^*$. Using the Littlewood-Richardson rule, we can decompose $\bigwedge^p E \otimes \bigwedge^p E^*$ into the following irreducible representation of $GL(E)$:
\[
\bigwedge^p E \otimes \bigwedge^p E^* = \dsum_{j=0}^{\min (p,n-p)} S_{(1^j,0^{n-2j},(-1)^j)} E.
\]
We denote the summand $S_{(1^j,0^{n-2j},(-1)^j)} E$ inside of $\bigwedge^p E \otimes \bigwedge^p E^*$ by $U_{j,p}$. Moreover, we denote the subspace of $\bigwedge^p E \otimes \bigwedge^p E^*$ equal to $U_{0,p} \dsum U_{1,p} \dsum \ldots \dsum U_{i,p}$ by $V_{i,p}$. In other words, if we fix two multi-indices $P = (p_1,\ldots,p_i)$ and $Q = (q_1,\ldots,q_i)$ and consider the sum of minors
\begin{equation} \label{eqn:v_i,p}
\sum_{| J | = p-i} (P,J|Q,J).
\end{equation}
Then $V_{i,p}$ is the span of such elements for different choices of $P$ and $Q$.

It is well known that the ideal of polynomial functions vanishing on nilpotent cone in $X$ is generated by the invariants $t_p = U_{0,p} = V_{0,p}$ of the action of $GL(n)$ on $X$. Here $t_p$ can be defined as the degree $(n-p)$-th coefficient of the characteristic polynomial of $\phi$.

In \cite{Weyman} Weyman proved the following theorem about a small but non minimal set of generators of $J_{\mu}$.

\begin{thm} \label{thm:generator}
The ideal $J_{\mu}$ is generated by the spaces $U_{i,\mu (i)}$ where $1 \leq i \leq n$ and $\mu (i) = \mu_1 + \ldots + \mu_i - i + 1$ (which are zero if $i > \min (\mu (i), n - \mu (i))$), and by the spaces $U_{0,p}$ (for $1 \leq p \leq \mu_1$) which correspond to the invariants $t_p$.
\end{thm}

The proof is based on the geometric methods in the paper \cite{Weyman} which we will simply describe.

\subsection{Geometric Methods for Calculating Syzygies} \label{sec:geometric}

Let $\mu^T = (\mu_1^T,\ldots,\mu_t^T)$ be the partition conjugate to $\mu$. Let $G/P_{\mu^T}$ be the partial flag variety whose points are given by partial flags $(R_1 \subset \ldots \subset R_t = E )$ with $\dim R_i = \mu_1^T + \ldots + \mu_i^T$, and 
\[
Y_{\mu} = \{ (\phi,R_1,\ldots,R_t) \in X \times G/P_{\mu^T} \mid \phi (R_i) \subset R_{i-1} \mbox{ for } i = 1, \ldots, t  \}.
\]
Then $Y_{\mu}$ is a desingularization of $X_{\mu}$ by the first projection map. We denote the tautological vector bundle on $G/P_{\mu^T}$ of dimension $\mu_1^T + \ldots + \mu_i^T$ by $\cR_i$ and its corresponding quotient bundle by $\cQ_i$. Since $Y_{\mu}$ is the total space of the cotangent bundle on $G/P_{\mu^T}$, if we denote by $\cT_{\mu}$ the tangent bundle on $G/P_{\mu^T}$, then we have the following short exact sequence of vector bundles on $G/P_{\mu^T}$
\[
0 \longrightarrow \cS_{\mu} \longrightarrow E \otimes E^* \otimes \cO_{G/P_{\mu^T}} \longrightarrow \cT_{\mu} \longrightarrow 0.
\]
In terms of tautological bundles $\cS_{\mu}$ can be described as
\[
\cS_{\mu} = \cR_1 \otimes \cQ_0^* + \cR_2 \otimes \cQ_1^* + \ldots + \cR_t \otimes \cQ_{t-1}^*
\]
where $\cR_t = E \times G/P_{\mu^T}$ and $\cQ_0 = E^* \times G/P_{\mu^T}$. Note that the sum here is not direct sum and the summands are generally not subbundles of $S_{\mu}$. 

All the above constructions are compatible with base change. Hence it makes sense to replace $E$ by a locally free sheaf $\cE$ on a scheme and consider all the constructions in the relative situation.

Applying the geometric method in \cite[Theorem 5.1.2]{weyman1} and \cite[Theorem 5.1.3]{weyman1} to this desingularization, we have the following theorem

\begin{thm}
The varieties $X_{\mu}$ are normal, with rational singularities. Furthermore, the minimal free resolution of the coordinate ring $K[X_{\mu}]$ as an $A$-module has terms
\[
\ldots \longrightarrow F_i^{\mu} \longrightarrow \ldots \longrightarrow F_1^{\mu} \longrightarrow F_0^{\mu} \longrightarrow 0
\]
where
\[
F_i^{\mu} = \dsum_{j \geq 0} H^j (G/P_{\mu^T},\bigwedge^{i+j} \cS_{\mu}) \otimes_K A(-i-j)
\]
where $A(-d)$ is the homogeneous free $A$-module with a single generator of degree $d$.
\end{thm}

Let $\hat{\mu}$ be the partition $\hat{\mu} = (\max(\mu_1 - 1,0),\max(\mu_2 -1,0),\ldots,\max(\mu_s-1,0))$. Then $\hat{\mu}^T = (\mu_2^T,\ldots,\mu_t^T)$. We will have the following short exact sequence
\[
0 \longrightarrow \cR_1 \otimes E^* \longrightarrow \cS_{\mu} \longrightarrow \cS_{\hat{\mu}} (\cQ_1,\cQ_1^*) \longrightarrow 0
\]
where $\cS_{\hat{\mu}} (\cQ_1,\cQ_1^*)$ is the bundle $\cS_{\hat{\mu}}$ constructed in a relative situation with $\cQ_1$ replacing $E$.

In order to get minimal set of generators of $X_{\mu}$, we can estimate the terms of $F_{\bullet}^{\mu}$ inductively. We do induction on $\mu_1$.

Consider the minimal free resolution of $F^{\hat{\mu}} (\cQ_1,\cQ_1^*)_{\bullet}$ in the relative situation with $\cQ_1$ replacing $E$. And we view this resolution as the sequence of $Sym(\cQ_1 \otimes \cQ_1^*)$-modules. To give an upper bound on the terms in $F_{\bullet}^{\mu}$, we look for all representations $S_{\lambda} \cQ_1$ in $F^{\hat{\mu}} (\cQ_1,\cQ_1^*)_{\bullet}$ and compute the push down $K_{\lambda} (E,E^*)$ of the complex $S_{\lambda} \cQ_1 \otimes \bigwedge^{\bullet} (\cR_1 \otimes E^*)$.

We can find the following theorem in section 3 of \cite{Weyman}.

\begin{thm}
Let $S_{\lambda} \cQ_1$ be the term in $F^{\hat{\mu}} (\cQ_1,\cQ_1^*)_{\bullet}$ of homological degree $i$. Then all terms coming from $K_{\lambda} (E,E^*)$ are in homological degree $\geq i$. So they can only give contribution to the terms $F_j^{\mu}$ where $j \geq i$. 
\end{thm}

According to this theorem, we have the following lemma.

\begin{lem} \label{lem:contribution}
The term $F_0^{\mu}$ in $F^{\mu}_{\bullet}$ consists of a single copy of $A$ in homogeneous degree $0$. By \cite[Theorem 3.12]{Weyman}, the term $F_1^{\mu}$ in $F^{\mu}_{\bullet}$ consists of two types:

\begin{enumerate}
    \item We have only the trivial representation in $F^{\hat{\mu}}(\cQ_1,\cQ_1^*)_{0}$. The corresponding push down $K_{(0)}$ of $\bigwedge^{\bullet} (\cR_1 \otimes E^*)$ gives a possible contribution to $F_1^{\mu}$ which is the vanishing of $n - \mu_1 + 1$ minors of the matrix $\Phi$.
    \item For each term $S_{\lambda} \cQ_1$ in $F^{\hat{\mu}} (\cQ_1,\cQ_1^*)_{1}$, it gives a possible contribution of the zero'th term of the corresponding complex $K_{\lambda}$. For the term $S_{(1^i,0^{n - \mu^T_1 - 2i},(-1)^i)} \cQ_1$ in homogeneous degree $\hat{\mu}(i)$ and homological degree 1 in particular, the image of its contribution in $F_0^{\mu}$ is contained in the ideal
    \[
    <U_{0,1},\ldots,U_{0,\mu_1},U_{1,\mu(1)},\ldots,U_{i,\mu(i)}> = <U_{0,1},\ldots,U_{0,n},V_{1,\mu(1)},\ldots,V_{i,\mu(i)}>.
    \]
\end{enumerate}
\end{lem}

\begin{section}{Proof of Main Theorem} \label{sec:main}
In the proof of the main theorem, we will use the following two lemmas whose proof could be found in \cite{weyman2}.

\begin{lem} \label{lem:rel}
If char $K = 0$, then $V_{i,p} = \dsum_{j = 0}^{min(i,n-p)} U_{j,p}$.
\end{lem}

\begin{lem} \label{lem:rel1}
Let $i \geq 1$, then the space $V_{i,p+1}$ is contained in the ideal generated by $V_{i,p}$.
\end{lem}

Now we can state our main theorem:

\begin{thm} \label{thm:main}
$J_{\mu}$ is minimally generated by $U_{0,p} (1 \leq p \leq \mu_1)$ and $U_{i,\mu (i)}$ for which the following condition is satisfied
\begin{equation} \label{eqn:main}
\mu(i) < \mu(i-1) + \lfloor \frac{\mu (i-1) - 1}{i - 1} \rfloor.
\end{equation}
\end{thm}

\begin{exmp}
The cases when $\mu_1 \leq 2$ or $\mu$ is a hook are previously known and they coincide with Theorem~{\ref{thm:main}}.
\end{exmp}

\begin{exmp}
Let $\mu = (3^2,2^2,1^5)$ and $\mu^T = (9,4,2) = $\ytableausetup{centertableaux}
\begin{ytableau}
\text{ } &  &  &  &  &  &  &  &  \\
 &  &  &  & \none[\uparrow] & \none[\uparrow] & \none[\uparrow] \\
 &  & \none[\uparrow] \\
 \none[\uparrow]
\end{ytableau}.

We label the column of $\mu^T$ by $1,2,\ldots,\mu_1^T$ starting from the leftmost column. Note that $\mu(i) = \mu_1 + \mu_2 + \ldots + \mu_i - (i-1)$ is the same as the total number of boxes in column $\leq i$ minus $(i-1)$. The arrow in column $i$ means exactly that $U_{i,\mu(i)}$ is in the minimal set of generators of $J_{\mu}$.  
In this example, $J_{\mu}$ is minimally generated by $U_{0,1},U_{0,2},U_{0,3},U_{1,3},U_{3,6},U_{5,7},U_{6,7},U_{7,7}$.
\end{exmp}

\begin{exmp}
If $\mu = (4,2^3,1^5)$ and $\mu^T = (9,4,1,1) = $\ytableausetup{centertableaux}
\begin{ytableau}
\text{ } &  &  &  &  &  &  &  &  \\
 &  &  &  & \none[\uparrow] & \none[\uparrow] & \none[\uparrow] \\
 & \none[\uparrow] & \none[\uparrow] \\
  \\ 
\none[\uparrow]
\end{ytableau}.
Then $J_{\mu}$ is minimally generated by $U_{0,1},U_{0,2},U_{0,3},U_{0,4},U_{1,4},U_{2,5},U_{3,6},U_{5,7},U_{6,7},U_{7,7}$.
\end{exmp}

Before the actual proof, we will show the following two propositions first.

\begin{prop} \label{lem:minimal}
If $\mu(i) < \mu(i-1) + \lfloor \frac{\mu(i-1) - 1}{i - 1} \rfloor$ and $U_{i,\mu(i)} \neq \{0\}$, then $U_{i,\mu(i)}$ is not contained in the ideal generated by set of equations $\{ U_{0,1}, \ldots, U_{0,\mu_1}, U_{1,\mu(1)}, \ldots, U_{i-1,\mu(i-1)} \}$ for any $\mu$ a partition of $n = \text{dim }E$.
\end{prop}

\begin{proof}
Assume $\mu=(\mu_1,\ldots,\mu_s)$. And define $e = \lfloor \frac{\mu(i-1) - 1}{i - 1} \rfloor$. So 
\begin{align} \label{def:mu(i-1)} 
\mu(i-1) - 1 = e(i-1) + f \mbox{ where } 0 \leq f \leq i-2. 
\end{align}
Since $U_{i,\mu(i)} \neq \{ 0 \}$, we have $i < s$. Also from the condition $\mu(i) < \mu(i-1) + \lfloor \frac{\mu(i-1) - 1}{i - 1} \rfloor$, we have $\mu_i < e+1$. 

Now we look at the partition 
\[\tilde{\mu} = (\underbrace{e+2,\ldots,e+2}_\text{$f$ copies},\underbrace{e+1,\ldots,e+1}_\text{$(i-1-f)$ copies},\mu_i + 1, \mu_{i+1}, \mu_{i+2}, \ldots, \mu_{s-2},\mu_{s-1},\mu_s-1).\]
We will restate and use \cite[Lemma 4.1]{Weyman} below:

\begin{lem} \label{lem:rel2}
When $i \geq 1$, $U_{i,p}$ vanish on $X_{\mu}$ iff $p \geq \mu(i)$.
\end{lem}

Note that by ~\ref{def:mu(i-1)}, $\mu_1 + \mu_2 + \ldots + \mu_{i-1} = \tilde{\mu}_1 + \tilde{\mu}_2 + \ldots + \tilde{\mu}_{i-1}$ since
\begin{align*}
& \mu_1 + \mu_2 + \ldots + \mu_{i-1} = \mu(i-1) + (i-2) = e(i-1) + f + (i-1) = (e+1)(i-1) + f \\
& \tilde{\mu}_1 + \tilde{\mu}_2 + \ldots + \tilde{\mu}_{i-1} = (e+2)f + (e+1)(i-1-f) = (e+1)(i-1) + f.
\end{align*}
On the other hand, we have $\mu_1 \geq \mu_2 \geq \ldots \geq \mu_{i-1}$. And \[\tilde{\mu}_1 = \tilde{\mu}_2 = \ldots = \tilde{\mu}_f = \tilde{\mu}_{f+1} + 1 = \tilde{\mu}_{f+2} + 1 = \ldots = \tilde{\mu}_{i-1} + 1.\]
Hence $\mu_1 + \mu_2 + \ldots + \mu_j \geq \tilde{\mu}_1 + \tilde{\mu}_2 + \ldots + \tilde{\mu}_j$ for all $1 \leq j \leq i-1$. And we have $\mu(j) \geq \tilde{\mu}(j)$ for all $1 \leq j \leq i-1$. Now according to Lemma~{\ref{lem:rel}}, Lemma~{\ref{lem:rel1}} and the fact that $\mu(j) \geq \tilde{\mu}(j)$ for all $1 \leq j \leq i-1$, we have the ideal 
\begin{align*}
<U_{0,1},\ldots,U_{0,\mu_1},U_{1,\mu(1)},\ldots,U_{i-1,\mu(i-1)} > &= \;  <U_{0,1},\ldots,U_{0,\mu_1},V_{1,\mu(1)},\ldots,V_{i-1,\mu(i-1)}> \\
&\subseteq \; <U_{0,1},\ldots,U_{0,n},V_{1,\tilde{\mu}(1)},\ldots,V_{i-1,\tilde{\mu}(i-1)}> 
\end{align*}
vanish on a generic matrix in $X_{\tilde{\mu}}$. Here $n$ is the dimension of $E$.

On the other hand, by Lemma~{\ref{lem:rel2}}, $U_{i,\mu(i)}$ do not vanish on $X_{\tilde{\mu}}$ since $\mu(i) < \tilde{\mu}(i)$. Therefore we have
$U_{i,\mu(i)} \nsubseteq \; <U_{0,1},\ldots,U_{0,\mu_1},U_{1,\mu(1)},\ldots,U_{i-1,\mu(i-1)}>$.
\end{proof}

\begin{exmp}
Let $\mu = (4,2^3,1^5)$ and $\mu^T = (9,4,1,1) = $\ytableausetup{centertableaux}
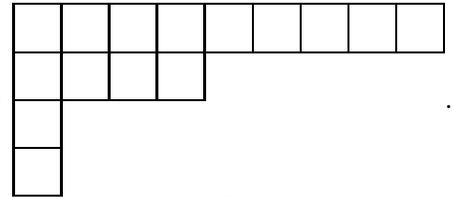
\begin{ytableau}
\text{ } &  &  &  &  &  &  &  &  \\
 &  &  &  \\
 \\
 \\
\end{ytableau}.
In order to show $U_{3,6} \nsubseteq \;  <U_{0,1},U_{0,2},U_{0,3},U_{0,4},U_{1,4},U_{2,5}>$, we look at $\tilde{\mu} = (3^3,2,1^4)$ and $\tilde{\mu}^T = (8,4,3) = $\ytableausetup{centertableaux}
\begin{ytableau}
\text{ } &  &  &  &  &  &  &  \\
 &  &  & \\
 &  &
\end{ytableau}. Then the ideal $<U_{0,1},U_{0,2},U_{0,3},U_{0,4},U_{1,4},U_{2,5}>$ vanish on $X_{\tilde{\mu}}$ but $U_{3,6}$ do not.
\end{exmp}

The proof of the second part is similar to \cite[Section 1]{weyman2}.

\begin{prop} \label{prop:generate}
If $\mu(i) \geq \mu(i-1) + \lfloor \frac{\mu(i-1) - 1}{i - 1} \rfloor$, then we have
\[
U_{i,\mu(i)} \subset \; <U_{0,1},\ldots,U_{0,\mu_1},U_{1,\mu(1)},\ldots,U_{i-1,\mu(i-1)}>.
\]
\end{prop}

\begin{proof}
Since $\mu(i) \geq \mu(i-1) + \lfloor \frac{\mu(i-1) - 1}{i - 1} \rfloor$, we have
\begin{align} \label{condition:generates}
 \mu_i - 1 \geq \lfloor \frac{\mu(i-1) - 1}{i - 1} \rfloor > \frac{\mu(i-1) - 1}{i - 1} - 1. 
\end{align}
Define 
\begin{equation} \label{def:tildeMu}
\tilde{\mu} = (\mu_1,\ldots,\mu_{i-1},\mu_i^r,f) \text{ where } n = \mu_1 + \ldots + \mu_{i-1} + r \mu_i + f \text{ and } 0 \leq f \leq \mu_i - 1. 
\end{equation}
We will show the following lemma
\begin{lem}
$J_{\tilde{\mu}}$ is generated by $\{ U_{0,1},\ldots,U_{0,\mu_1},U_{1,\mu(1)},\ldots,U_{i-1,\mu(i-1)} \}$.
\end{lem}

\begin{proof}
We will do induction on $\mu_i$. We can look at 
\[\mu ' = (\mu_1 - (\mu_i - 1),\mu_2 - (\mu_i-1),\ldots,\mu_{i-1} - (\mu_i - 1),1^r).\] 
Then 
\[ \mu ' (i) = \mu_1 + \ldots + \mu_{i-1} - (i-1)(\mu_i - 1) + 1 - i + 1 = \mu(i-1) - 1 - (i-1)(\mu_i - 1) + 1 < i \]
where the last inequality is due to ~\ref{condition:generates}. Hence $U_{i,\mu ' (i)} = \{0\}$. And by Theorem~{\ref{thm:generator}}, $J_{\mu '}$ is generated by $\{ U_{0,1},\ldots,U_{0,\mu^{\prime}_1},U_{1,\mu '(1)},\ldots,U_{i-1,\mu '(i-1)} \}$. When $\mu_i = 1$, $\mu ' = \mu$ and the conclusion holds. Let $\hat{\tilde{\mu}} = (\mu_1 - 1,\ldots,\mu_{i-1} - 1, (\mu_i - 1)^r,\text{max}(f-1,0))$. By the induction hypothesis, we can assume $J_{\hat{\tilde{\mu}}}$ is generated by $\{ U_{0,1},\ldots,U_{0,\hat{\tilde{\mu}}_1},U_{1,\hat{\tilde{\mu}}(1)},\ldots,U_{i-1,\hat{\tilde{\mu}}(i-1)} \}$. According to ~\ref{lem:contribution}, we have $J_{\tilde{\mu}}$ is generated by 
\[ \{ U_{0,1},\ldots,U_{0,\tilde{\mu}_1},U_{1,\tilde{\mu}(1)},\ldots,U_{i-1,\tilde{\mu}(i-1)} \} = \{ U_{0,1},\ldots,U_{0,\mu_1},U_{1,\mu(1)},\ldots,U_{i-1,\mu(i-1)} \} \]
and $(n - \tilde{\mu}^T_1 + 1) \times (n - \tilde{\mu}^T_1 + 1)$ minors of $\Phi$. Here $\tilde{\mu}$ is defined in \ref{def:tildeMu}. 

Let $k = \tilde{\mu}^T_1 = l(\tilde{\mu})$, we have $\tilde{\mu}(k) = n - \tilde{\mu}^T_1 + 1$. But $k + \tilde{\mu}(k) = n + 1 > n$. So $U_{k,\tilde{\mu}(k)} = \{0\}$. If $f \neq 1$, then $\tilde{\mu}(k-1) = \tilde{\mu}(k) - \mu_k + 1 < n - \tilde{\mu}^T_1 + 1$. Hence by Theorem ~\ref{thm:generator}, the only possible generator in degree $(n - \tilde{\mu}^T_1 + 1)$ is the invariant $U_{0,n - \tilde{\mu}^T_1 + 1}$. This is contained in $V_{1,n - \tilde{\mu}^T_1 + 1}$ which is contained in the ideal generated by $\{U_{0,\tilde{\mu}_1},U_{1,\tilde{\mu}_1} = U_{1,\tilde{\mu}(1)} \}$ according to Lemma~{\ref{lem:rel}} and Lemma~{\ref{lem:rel1}}. 

However, when $f = 1$, $\tilde{\mu}(k-1) = n - \tilde{\mu}^T_1 + 1$ and $\tilde{\mu}(k-2) < n - \tilde{\mu}^T_1 + 1$. So we only need to show that $V_{k-1,\tilde{\mu}(k-1)}$ is contained in $< U_{0,1},\ldots,U_{0,\mu_1},U_{1,\mu(1)},\ldots,U_{i-1,\mu(i-1)}> =  < U_{0,1},\ldots,U_{0,\mu_1},V_{1,\mu(1)},\ldots,V_{i-1,\mu(i-1)} >$. Let $m = n + \mu_i - 2 \geq n$ and look at
\[
\mu '' = (\mu_1,\mu_2,\ldots,\mu_{i-1},\mu_i^r,\mu_i)
\]
which is a partition of $m$. Note that this only differs from $\tilde{\mu}$ in \ref{def:tildeMu} in the $k$-th entry. But this falls into the situation of the previous paragraph with $f = 0$. Hence we have $V_{k-1,\tilde{\mu}(k-1)}$ is generated by $\{ U_{0,1},\ldots,U_{0,\mu_1},V_{1,\mu(1)},\ldots,V_{i-1,\mu(i-1)} \}$ for spaces of $m \times m$ matrices by previous argument. Now notice that using a basis of $V_{i,p}$ given by \ref{eqn:v_i,p}, we can get a basis of $V_{i,p}$ for space of $n \times n$ matrices with $n \leq m$ by taking a basis of $V_{i,p}$ for space of $m \times m$ matrices and restricting them on the variables $\{ \phi_{i,j} \} \mid_{1 \leq i,j \leq n}$. Hence we also have $V_{k-1,\tilde{\mu}(k-1)}$ is generated by $\{ U_{0,1},\ldots,U_{0,\mu_1},V_{1,\mu(1)},\ldots,V_{i-1,\mu(i-1)} \}$ for spaces of $n \times n$ matrices.
\end{proof}

By Lemma~{\ref{lem:rel2}}, $U_{i,\mu(i)} = U_{i,\tilde{\mu}(i)}$ is contained in $J_{\tilde{\mu}} = <U_{0,1},\ldots,U_{0,\mu_1},U_{1,\mu(1)},\ldots,U_{i-1,\mu(i-1)}>$. This concludes the proof.
\end{proof}

\begin{proof}[Proof of Theorem~{\ref{thm:main}}]
On the one hand, by Lemma~{\ref{lem:minimal}}, we see that the set of equations described in Theorem~{\ref{thm:main}} is a minimal set of equations, i.e., they do not generate each other. On the other hand, using Proposition~{\ref{prop:generate}} inductively on $i$, we conclude that using only the set of equations in Theorem~{\ref{thm:main}}, we can generate all equations described in Theorem~{\ref{thm:generator}}. This concludes the whole proof.
\end{proof}

\begin{rmk}
In \cite{ES}, Eisenbud and Saltman generalized the notion of nilpotent orbit closure to rank varieties. Given a partition $\mu$ of $l$ with $l \leq n$, we can consider the variety consisting of all endomorphisms $\phi \colon E \rightarrow E$ such that ker $\phi^i \geq \mu^T_1 + \ldots \mu^T_i$ for all $i$. The whole proof is also valid for rank varieties. And the minimal set of generators defining the rank variety scheme theoretically is given by $U_{0,p}$($n - l + 1 \leq p \leq n - l + \mu_1$) and $U_{i,\mu(i)}$ for which the Equation~{\ref{eqn:main}} holds.
\end{rmk}

\end{section}

\small \noindent Hang Huang, Department of Mathematics,
University of Wisconsin Madison, Madison, WI 53706 \\
{\tt hhuang235@math.wisc.edu}, \url{http://math.wisc.edu/~hhuang235/}

\end{document}